\documentclass[10pt]{amsart}

\usepackage{amsmath}
\usepackage{amssymb}
\usepackage[shortalphabetic]{amsrefs}
\usepackage{stmaryrd}
\usepackage{verbatim}
\usepackage{enumerate}
\usepackage[all]{xy}
\usepackage{tikz}

\usetikzlibrary{shapes.geometric,arrows}
\tikzstyle{startstop} = [rectangle, rounded corners, minimum width=3cm, minimum height=1cm,text centered, draw=black, fill=white]
\tikzstyle{io} = [trapezium, trapezium left angle=70, trapezium right angle=110, minimum width=3cm, minimum height=1cm, text centered, draw=black, fill=white]
\tikzstyle{process} = [rectangle, minimum width=2cm, minimum height=1cm, text centered, text width=3.5cm, draw=black, fill=white]
\tikzstyle{decision} = [diamond, minimum width=2cm, minimum height=2cm, text centered,  text width=2cm, draw=black, fill=white]
\tikzstyle{arrow} = [thick,->,>=stealth]

\newtheorem{thm}{Theorem}
\newtheorem*{thm*}{Theorem}

\newtheorem{prop}[equation]{Proposition}

\newtheorem{coro}[equation]{Corollary}

\theoremstyle{remark}
\newtheorem*{remark*}{Remark}

\theoremstyle{definition}

\DeclareMathOperator{\End}{End}

\DeclareMathOperator{\Adj}{Adj }
\DeclareMathOperator{\Der}{Der }

\DeclareMathOperator{\GL}{GL}

\DeclareMathOperator{\Aut}{Aut}

\DeclareMathOperator{\Cent}{Cen}

\newcommand{\bmto}{\rightarrowtail}

\title{New Lie products for groups and their automorphisms}

\author{James B. Wilson}
\address{
	Department of Mathematics\\
	Colorado State University\\
	Fort Collins, CO 80523\\
}
\email{James.Wilson@ColoState.Edu}
\date{\today}

\begin{document}

\maketitle
\begin{abstract}
We generalize the common notion of descending and ascending central series.
The descending approach determines a naturally graded Lie ring and
the ascending version determines a graded module for this ring.  We also link derivations of these rings to the automorphisms
of a group.  This process uncovers new structure in 4/5 of the approximately 11.8 million groups of size at most 1000 and beyond that point pertains
to at least a positive logarithmic proportion of all finite groups. 
\end{abstract}

\section{Introduction}

At the 1958 International Congress of Mathematicians Graham Higman presented 
``Lie ring methods in the theory of finite nilpotent groups'' \cite{Higman:Lie}. He promoted a connection between 
Group Theory and Lie Theory relying on the general observations that group commutators $[x,y]=x^{-1} x^y=x^{-1}y^{-1}xy$ obey 
relations that mimic the distributive,
alternating, and the Jacobi identities, 
\begin{align*}
	[xy,z]  = [x,z]^y [x,z],\qquad
	[x,x]  = 1,\qquad
	[[x,y^{-1}],z]^{y}[[y,z^{-1}],x]^z[[z,x^{-1}],y]^x  = 1.
\end{align*}
Already Lazard and Mal'cev had exploited this and even more subtle connections; see \cite{Khukhro}*{Chapter 3} for an excellent modern survey. 
Higman's point was to popularize 
connections to Lie rings ``...less precise,  but much more generally applicable'' \cite{Higman:Lie}*{p. 307}.  
The spirit of the present article is similar.

We generalize the descending and ascending central series so that commutation products exist between the
factors and make Lie rings and modules, but we relax
the condition that they be graded by positive integers (Theorems~\ref{thm:ascending-Lie-basic} \& \ref{thm:layers}).  We further lift this correspondence to automorphism groups (Theorems~\ref{thm:auto-basic} \& \ref{thm:auto}).
By {\em commutation products} we mean those products between factors
with the following operations ($\bar{x}$ denotes a congruence class):
\begin{align*}
	\bar{x}\circ\bar{y} & = \overline{[x,y]} & \bar{x}+\bar{y}& = \overline{xy} & -\bar{x} & = \overline{(x^{-1})}
		& 0 & = \bar{1}.
\end{align*}
Commutation products that are graded by cyclic semigroups we call {\em classical}.  

We know of five easily computed non-classical commutation products, three introduced in \cite{Wilson:alpha} and two here.  
 On a log-scale a positive proportion of all finite groups admit a non-classical
commutation product \cite{Wilson:alpha}.  Thanks to the recent algorithms of Maglione (see \cite{Maglione:filter}), we conducted a test to 
find the actual proportion amongst small groups \cite{millennium}.  
We found at least 9,523,263 of the 11,758,814 
groups (81\%) of order at most 1000 have new non-classical commutation products;
Figure~\ref{fig:data} plots this proportion. It remains a mystery why a simple
relaxation of subscripts discovers so much structure but we offer some clues
in Section~\ref{sec:closing}.

\begin{figure}[!htbp]
\input{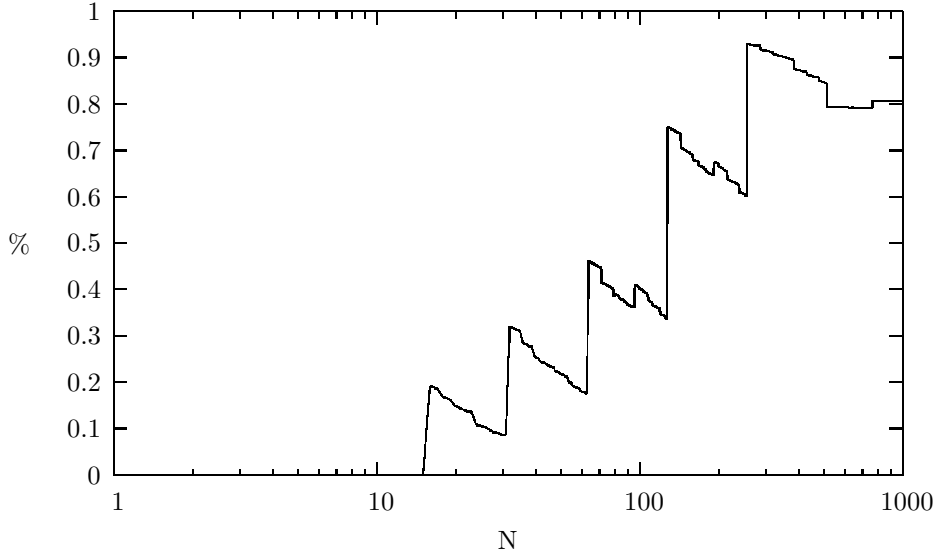}
\caption{The proportion of groups of order $N\leq 1000$ with newly found characteristic non-classical commutation Lie products.}\label{fig:data}
\end{figure}

 To begin piecing together the origins of these new characteristic structures we rethink the meaning of central ``series'' keeping the idea of
ascending and descending and associated commutation products.
\smallskip

Usually when thinking of commutation products we start 
with the {\em lower central series} where
$\gamma_1=G$ and $\gamma_{i+1}=[G,\gamma_i]$,
and the {\em upper central series} where $\zeta^0=1$ 
and $\zeta^{i+1}$ is the centralizer of $G/\zeta^i$.  Commutation makes 
 a natural graded Lie ring $L_*(\gamma)=\bigoplus_{i> 0}\gamma_i/\gamma_{i-1}$; cf. \cite{Khukhro}*{Chapter 3}.
Perhaps known but certainly under used is the following Lie product for the upper central series.

\begin{thm}\label{thm:ascending-Lie-basic}
Given any group $G$, a commutative ring $R$, and an $R$-module $Q$,
\begin{align*}
	L^*(\zeta;Q) & = \bigoplus_{i>0}  \hom_{\mathbb{Z}}(\zeta^{i}/\zeta^{i-1},Q)
\end{align*}
is a natural graded Lie module of the natural Lie ring $L_*(\gamma)\otimes_{\mathbb{Z}} R$.
\end{thm}

For torsion free groups the use of $R=Q=\mathbb{Q}$ is a natural choice.
For finite groups, if we use $R=\mathbb{Z}$ and $Q=\mathbb{Q}/\mathbb{Z}$
then $\hom(\zeta^i/\zeta^{i-1},\mathbb{Q}/\mathbb{Z})\cong \zeta^i/\zeta^{i-1}$,
and for $p$-groups we might prefer $Q=\mathbb{Z}_{p^e}$ or $\mathbb{Z}_{p^{\infty}}$.
So we prove:
\begin{coro}\label{coro:1}
For finite groups, the upper central series factors  $L^i(\zeta)=\zeta^{i}/\zeta^{i-1}$ form  
a  graded Lie module $L^*(\zeta)  = \bigoplus_{i>0}  L^i(\zeta)$ for the graded Lie ring $L_*(\gamma)=\bigoplus_{i>0} L_i(\gamma)$ 
of the lower central series factors $L_i(\gamma)=\gamma_i/\gamma_{i+1}$.
\end{coro}

We will make both these results far more general below.  Our further connection between central factors and Lie algebras relates to the structure of automorphism groups.
In general we expect no relationship to exist between the central series of a group and its automorphism group 
(consider $G=\mathbb{Z}_p^n$ and $\Aut(G)\cong \GL(d,p)$). 
However,  we argue the more appropriate Lie algebra to explore in connection to $\Aut(G)$  is the graded
derivation ring $\Der(L(\gamma))$.  A hint of this might be recognized from the family of
automorphisms $\phi$ on $G$ that are the identity on $G/\zeta^1$.  Such ``central'' automorphisms induce
a linear map $G/\gamma_2\to \zeta^1$ by $\gamma_2 x\mapsto x^{-1} x\phi$.  The following theorem can be seen
as generalizing central automorphisms but retaining the vital Lie structure.  Stated just for the traditional 
central factors we prove:

\begin{thm}\label{thm:auto-basic}
Let $G$ be a group and $A\to \Aut(G)$ a group homomorphism.  Define
\begin{align*}
	A_i & = \{ a\in A : \forall j, \forall g\in \gamma_j(G),~ g^{-1} g^a\in \gamma_{i+j}(G)\} & i>0.
\end{align*}
Then for all $i,j>0$, $[A_i,A_j]\leq A_{i+j}$ and $L_*(A)=\bigoplus_{i>0} A_i/A_{i+1}$ is a graded Lie ring
with a canonical graded Lie homomorphism into the graded derivation ring $\Der(L_*(\gamma))$.  Also,
$A/A_1$ is the faithful representation of $A$ acting on $L_*(\gamma)$ by conjugation.
\end{thm}

Both $L(\gamma)$ and $\Der(L(\gamma))$ are objects
whose structure we can capture through linear algebra and related efficient algorithms as in \cite{deGraaf}.  Meanwhile even knowing
a generating set for $\Aut(G)$ seems impossible for most groups.  So in this light Theorem~\ref{thm:auto-basic} has great practical
value.
\smallskip

Theorem~\ref{thm:ascending-Lie-basic} and especially its corollary appear classical to us, though we have not
discovered it in the literature so far.  So we give a synopsis.  The mechanics came from a largely unrelated setting.  We 
imitate Whitney's cap-product from algebraic topology; cf. \cite{Hatcher}*{Section 3.3}.  
The cap-product is a mixture of homology and cohomology groups.  Likewise, commutation products involving
the upper central series factors mix the lower and upper central series factors as follows
\begin{align*}
	\circ:&L_i(\gamma)\times L^{i+j}(\zeta)\bmto L^{j}(\zeta)
\end{align*}
To make a module we ``shuffle'' this product using a technique introduced by Knuth and Liebler to create new semifields from
old ones \citelist{\cite{Knuth}\cite{Liebler}}.  For example, we create
\begin{align*}
	\circ:&L_i(\gamma)\times \hom(L^j(\zeta),\mathbb{Q}/\mathbb{Z})\bmto \hom(L^{i+j}(\zeta),\mathbb{Q}/\mathbb{Z}).
\end{align*}
When the $L^i(\zeta)$ are finite, $\hom(L^i(\zeta),\mathbb{Q}/\mathbb{Z})\cong L^i(\zeta)$ leading to Corollary~\ref{coro:1}, 
but these isomorphisms are not canonical so it is more natural to retain the coefficients.

\section{Generalized ascending and descending central factors}
We aim for a generalization of Theorems~\ref{thm:ascending-Lie-basic} and \ref{thm:auto-basic} that introduces 
Lie products for groups that are graded by monoids other than the positive integers.  The purpose is to find new structure.  
By permitting acyclic gradings we can first profit from all the known Lie methods with standard gradings, and then
add recursive tools to refine the rings and modules by increasing the grading parameter.
\smallskip

Throughout we fix $M=\langle M,+,0,\prec\rangle$
a pre-ordered commutative monoid which in particular requires that for all $s\in M$, $0\prec s$ and
for all $s,t,u,v\in M$, if $s\prec t$ and $u\prec v$ then $s+u\prec t+v$.  In general we think of $M=\mathbb{N}^n$
with either the point-wise or the lexicographic pre-order.
In a group $G$ commutators are $[g,h]=g^{-1} g^h=g^{-1}h^{-1}gh$ and for $X,Y\subseteq G$, 
$[X,Y]=\langle [x,y] : x\in X,y\in Y\rangle$.  We liberally assume the usual commutator calculus such as in \cite{Robinson}*{Chapter 5}.
\smallskip

Following \cite{Wilson:alpha}*{Section 3}, a {\em filter} $\phi:M\to 2^G$ on a group $G$
is a function on a pre-ordered
commutative monoid $M=\langle M, +,0,\prec\rangle$ 
into  the subsets $2^G$ of $G$ with the following properties.  Each $\phi_s$ is a subgroup and for all $s$ and $t$ in $M$, 
\begin{align}\label{def:filter}
	[\phi_s,\phi_t] & \leq \phi_{s+t}\qquad \&\qquad s\prec t  \Rightarrow \phi_s\geq \phi_t.
\end{align}
Thus the burden of explaining the structure of commutation is shared between the operation and the ordering on the monoid.
For instance, it is possible to give a total order, i.e. a series of subgroups, but use indices that come from a non-cyclic monoid 
so that the commutation between the series terms can be far more intricate than by in usual central series.

To every filter $\phi:M\to 2^G$ there is an associated  boundary filter $\partial\phi:M\to 2^G$ given by
$\partial_s \phi = \langle \phi_{s+t} : t\neq 0\rangle$.   Fix a commutative unital associative ring $R$ and define the following $R$-modules.
\begin{align*}
	L_0(\phi;R) & = 0 & 
		s\neq 0 \Rightarrow L_s(\phi;R) & = (\phi_s/\partial_s\phi)\otimes_{\mathbb{Z}} R.
\end{align*}
Every filter determines the following  $M$-graded Lie ring \cite{Wilson:alpha}*{Theorem 3.1}.
\begin{align*}
	L_*(\phi;R) & = \bigoplus_{s\in M} L_s(\phi;R)& (\bar{x}_s\otimes r)\circ (\bar{y}_t\otimes s) & = \overline{[x,y]}_{s+t}\otimes rs.
\end{align*}
Here, $\bar{x}_s$ denotes a coset in $L_s$, etc..  Set $L_*(\phi)=L_*(\phi;\mathbb{Z})$.  Actually, each $L_s(\phi;R)$
is a natural $R[\phi_0/\partial_0 \phi]$-module so the portion at the top of a filter can be regarded as storing operators.
\medskip

In a dual manner, a {\em layering} $\pi:M\to 2^G$ on a group $G$ is a function where each $\pi^s$ is a subgroup,
and for all $s$ and $t$ in $M$,
\begin{align}\label{def:layering}
	\left[\pi^s,\cap_{u\neq 0} \pi^{t+u}\right] & \leq \pi^t 
	& s\prec t & \Rightarrow \pi^s\leq \pi^t.
\end{align}
Form the boundary $\partial\pi:M\to 2^G$ in the reversed order as $\partial^s \pi = \cap_{t\neq 0} \pi^{s+t}$.
So, if $M=\mathbb{N}$ then  $\partial^i \pi=\pi^{i+1}$.  If $G=\pi^M=\langle \pi^s : s\in M\rangle$ then \eqref{def:layering} says simply that
$[G,\pi^{i+1}]\leq \pi^i$ for every $i$.  This will be recognized as Hall's condition for a series to be called ``central''.
Using general monoids, which may not be well-ordered, it is no longer possible to simply take successors to move up  hence
$\partial^s\pi$ is used instead.  Also, the flexibility to adjust $\pi^0$ and $\pi^M$ makes it possible to define such series in 
setting that are non-nilpotent but still related, e.g. in automorphism groups of nilpotent groups.

These definitions determine abelian groups as follows (proved in Section~\ref{sec:strata}).
\begin{align*}
	L^0(\pi) & = 0 & s\neq 0 \Rightarrow L^s(\pi) = \partial^s\pi/\pi_s.
\end{align*}
To state the main result we need  a version with coefficients.  For that we fix an $R$-module $Q$ and set
\begin{align*}
	L^s(\pi;Q) & = \hom_{\mathbb{Z}}(L^s(\pi), Q).
\end{align*}
While filters give natural ring structures, layerings produce natural modules for these rings.
A ring is always better understood with a module.

\begin{thm}\label{thm:layers}
For every filter $\phi:M\to 2^G$ and a layering $\pi:M\to 2^G$, if
\begin{align}\label{def:sift}
	(\forall & s,t\in M) & [\phi_s,\pi^{s+t}]  & \leq \pi^t,
\end{align}
then $L^*(\pi;Q)$ is a graded module of the $M$-graded Lie ring $L_*(\phi;R)$.
\end{thm}

 Observe that in the case where $\phi=\gamma$ and $\pi=\zeta$ we know
that for all $i,j>0$, $[\gamma_i,\zeta^{i+j}]\leq \zeta^j$.  So Theorem~\ref{thm:ascending-Lie-basic} is a consequence
of Theorem~\ref{thm:layers}.
\smallskip

The next important thrust is to push the effects of a filter into the automorphism group and thus
linearize various properties of automorphisms that once required more difficult inspection.
Fix a group $G$ and a group homomorphism $A\to \Aut(G)$.  For $x\in G$ and $a\in A$
we define $[x,a]  = x^{-1} x^a$.

\begin{thm}\label{thm:auto}
For a filter $\phi:M\to (2^G)^A$ (that is a filter of $A$-invariant subgroups), the function $\Delta\phi:M\to 2^A$ given by
\begin{align*}
	\Delta_s\phi & = \{ a\in A : \forall t\in M, [\phi_t,a]\leq \phi_{t+s}\}
\end{align*}
is a filter and there is a natural graded Lie ring homomorphism from $L_*(\Delta\phi)$ into the derivation Lie ring $\Der(L_*(\phi))$.
\end{thm}

Note Theorem~\ref{thm:auto-basic} follows from Theorem~\ref{thm:auto}.
 
\section{The strata $L^t(\pi;Q_R)$ and the action by $L_t(\phi;R)$}\label{sec:strata}
Since filters were covered in detail along with examples in \cite{Wilson:alpha}, we can forgo most of that and focus on 
the new aspects which concern the interplay with layerings.
We begin by confirming the claims that $L^s(\pi;R)$ is well-defined.  

Fix a layering $\pi:M\to 2^G$.  
The universe being described here is between the subgroup $\pi^0$ and
$\pi^M=\langle	\pi^s:s\in M\rangle$.  For every $s,u$, $s\prec s+u$ 
which implies that $\pi^s\leq \partial^s \pi=\cap_{u\neq 0} \pi^{s+u}$.
Hence, for all $s$,
\begin{align*}
	[\pi^M,\pi^s] & = [\pi^M,\partial^s \pi] 
		 \leq \prod_{t\in M} [\pi^t,\partial^s \pi]\leq \pi^s.
\end{align*}
Thus $\pi^s$ and $\partial^s\pi$ are normal in $\pi^M$ and 
$L^s(\pi)=\partial^s \pi/\pi^s$ is central in $\pi^M/\partial^s \pi$.  
In particular $L^s(\pi)$ is abelian.  Also notice that 
$\partial\pi:M\to 2^G$, and as $\pi$ is order preserving so too is
$\partial\pi$.  Finally,
\begin{align*}
	[\pi^M,\partial^s(\partial\pi)] & =
		\bigcap_{t\neq 0} [\pi^M, \partial^{s+t}\pi]
		=\bigcap_{t,u\neq 0} [\pi^M, \pi^{s+t+u}]
		\leq \bigcap_{t\neq 0}\pi^{s+t}=\partial^s\pi.
\end{align*}
In particular, $\partial\pi$ is a layering as well.

Unlike the case of filters there is no immediate commutation product to impose on 
$L^*(\pi)$.  Nevertheless, we recognize within these ingredients the mechanics similar to the cap-product in
algebraic topology; cf \cite{Hatcher}*{Section 3.3}.  We make up a product of a similar shape.
\medskip

Straining our metaphors we say that a filter $\phi:M\to 2^G$ {\em sifts} a layering $\pi:M\to 2^G$ if
\eqref{def:sift} holds.

\begin{prop}
If a filter $\phi:M\to 2^G$ sifts a layering $\pi:M\to 2^G$ then for $s\neq 0$,
$\circ:L_s(\phi)\times L^{s+t}(\pi)\bmto L^t(\pi)$ defined by
\begin{align*}
		& \bar{x}\circ \bar{y} \equiv [x,y]\pmod{\pi^t}
\end{align*}
is well-defined and biadditive.
\end{prop}
\begin{proof}
Under the assumptions $\phi$ also sifts $\partial\pi$ and $\partial\phi$ sifts $\pi$:
\begin{align}\label{eq:phi-partial}
	[\phi_s,\partial^{s+t}\pi] & 
		= \bigcap_{u\neq 0}[\phi_s, \pi^{s+t+u}] 
		\leq \bigcap_{u\neq 0} \pi^{t+u}=\partial^{t}\pi.\\
	[\partial_s \phi,\pi^{s+t}] 
	& = \langle [\phi_{s+u},\pi^{s+t}] : u\neq 0\rangle
	 \leq \langle [\phi_{s+u},\pi^{s+t+u}] : u\neq 0\rangle
	 \leq \pi^s.
\end{align}
This means that the commutation map $\phi_s\times\partial^{s+t}\pi\to \partial^t\phi$ factors through
$\circ:L_s(\phi)\times L^{s+t}(\pi)\to L^t(\phi)$.

Now we prove biadditivity of $\circ$.  For all $x,y\in \phi_s$ and all $z\in \partial^{s+t}\pi$,
$[[x,z],y]\in [[\phi_s,\partial^{s+t}\pi],\phi_s]\leq [\partial^t\pi,\phi_s]$.
As $s\prec s+t$,
$[\phi_s,\partial^t\pi]\leq [\phi_s,\partial^{s+t}\pi]\leq \pi^t$.  Hence 
\begin{align*}
	(\bar{x}+\bar{y})\circ \bar{z}
		& \equiv [x,z][[x,z],y][y,z]
		 \equiv \bar{x}\circ \bar{z}+\bar{y}\circ\bar{z}\pmod{\pi^t}.
\end{align*}
Also for $x\in \phi_s$, $y,z\in \partial^{s+t}\pi$, $[[x,y],z]\in [[\phi_s,\partial^{s+t}\pi],\partial^{s+t}\pi]\leq [\partial^{t}\pi,\pi^M]\leq \phi^t$. So,
\begin{align*}
	\bar{x}\circ (\bar{y}+\bar{z}) 
		& \equiv  [x,z][x,y][[x,y],z]
		 \equiv  \bar{x}\circ \bar{y}+\bar{x}\circ \bar{y} \pmod{\pi^t}.
\end{align*}
\end{proof}

There is of course a similar map $L^{s+t}(\pi)\times L_s(\phi)\bmto L^t(\pi)$ 
but this is nothing more than the negative transpose of
the map above since $[x,y]^{-1}=[y,x]$.  A more pressing concern is the awkwardness of the indices in the product -- it is not graded but nearly.  
If our monoid has cancellation we could pass to the associated Grothendieck group and rewrite the grading as $L^s\times L_t\bmto L^{t-s}$, but
still we would need to clarify at each stage that $t-s\in M$ in order that the product be defined.  We promote a remedy that makes the resulting algebra more natural at the cost of introducing coefficients.
\medskip

Suppose we have an arbitrary bimap (biadditive map) $*:U\times V\bmto W$ and an abelian group $Q$.  We create
a second bimap $\#:U\times \hom_{\mathbb{Z}}(W,Q)\bmto \hom_{\mathbb{Z}}(V,Q)$ with the following definition:
\begin{align}
	(\forall & u\in U, \forall f\in\hom_{\mathbb{Z}}(W,Q)) & 	
		(u\# f) &: v\mapsto (u*v)f.
\end{align}
Such shuffling of products were introduced for nonassociative division rings (semifields) by Knuth and reformulated in a basis free way 
by Liebler \citelist{\cite{Knuth}\cite{Liebler}}.  Though Knuth called these ``transposes'', we use the less overloaded term of ``shuffle'' as suggested by
Uriya First. Knuth-Liebler shuffles are  the subject of an ongoing study of bimap categories by First and the author \cite{FW}.
Here we use them simply to configure the objects above into familiar vocabulary of Lie modules.
\medskip

Now let $Q=Q_R$ be an $R$-module for a commutative associative unital ring $R$.
As defined $L^s(\pi;Q)=\hom_{\mathbb{Z}}(L^s(\pi),Q)$.  When we apply a Knuth-Liebler shuffle to the products
 $\circ:L_s(\phi)\times L^{s+t}(\pi)\bmto L^{t}(\pi)$ we arrive at a new product
\begin{align*}
	\cdot:L_s(\phi;R)\times L^t(\pi;Q)\bmto L^{s+t}(\pi;Q)
\end{align*}
where for $\bar{x}\in L_s(\phi)$, $r\in R$, $f:L^t(\pi)\to Q\in L^t(\pi;Q)$, $(\bar{x}\otimes r)\cdot f$ is the following map
$L^{s+t}(\pi)\to Q$:
\begin{align}\label{def:module}
	 (\bar{x}\otimes r)\cdot f&:\bar{y}\mapsto -(\bar{x}\circ\bar{y})f r.
\end{align}
This new product is $R$-bilinear and further permits us to take the direct sum over all $s,t\in M$ to create a single graded $R$-bimap
$L_*(\phi;R)\times L^*(\pi;Q)\bmto L^*(\pi;Q)$.  It remains to show this makes $L^*(\pi;Q)$ into a Lie module for $L_*(\phi;R)$.

\begin{prop}\label{prop:Lie-module}
For all $\bar{x}\in L_s(\phi)$, all $\bar{y}\in L_t(\phi)$, and all $\bar{z}\in L^{s+t+u}(\pi)$, 
\begin{align*}
	[\bar{x},\bar{y}]\circ \bar{z} 
	\equiv \bar{x}\circ (\bar{y}\circ\bar{z})-\bar{y}\circ (\bar{x}\circ \bar{z})\pmod{ \pi^u}.
\end{align*}
\end{prop}

\begin{proof}
From the Hall-Witt identity it follows that for all $x,y,z$:
\begin{align*}
	[[x,y],z] = [y^{-1},[x^{-1},z]^{-1}][[y^{-1},[x^{-1},z]^{-1}],xy] [x,[y^{-1},z^{-1}]] [ [x,[y^{-1},z^{-1}]],zy].
\end{align*}
Fix $x\in \phi_s$, $y\in \phi_t$, $z\in \partial^{s+t+u} \pi$.  By \eqref{def:sift} and \eqref{eq:phi-partial} we note the following containments.
\begin{align*}
	[[y^{-1},[x^{-1},z]^{-1}],xy] & \in [[\phi_t,[\phi_s,\partial^{s+t+u}\pi],\phi_s\phi_t]
		\leq [[\phi_t,\partial^{t+u}\pi],\phi_0]\leq [\pi^{u+0},\phi_0]\leq \pi^u\\
	[[x,[y^{-1},z^{-1}]],zy] & \in [[\phi_s,[\phi_t,\partial^{s+t+u}\pi]],\partial^{s+t+u}\pi\phi_0]\leq [\pi^u,\pi^M][\pi^u,\phi_0]\leq \pi^u.
\end{align*}  
Hence, in the original product $\circ$ we find the following identity.
\begin{align*}
	[\bar{x},\bar{y}]\circ \bar{z} & \equiv [[x,y],z] \pmod{\pi^u}\\
		&\equiv [y^{-1},[x^{-1},z]^{-1}][[y^{-1},[x^{-1},z]^{-1}],xy] [x,[y^{-1},z^{-1}]] [ [x,[y^{-1},z^{-1}]],zy]\\
		& \equiv [y^{-1},[x^{-1},z]^{-1}][x,[y^{-1},z^{-1}]]\\
		& \equiv \bar{x}\circ(\bar{y}\circ \bar{z})-\bar{y}\circ (\bar{x}\circ \bar{z}).
\end{align*}
\end{proof}

\section{Proof of Theorem~\ref{thm:layers}}
To complete the proof of Theorem~\ref{thm:layers} we must now consider Proposition~\ref{prop:Lie-module}
in the presence of the appropriate Knuth-Liebler shuffle.

Fix $\bar{x}\in L_s(\phi)$, $\bar{y}\in L_t(\phi)$, $r,s\in R$, and $f:L_u(\pi)\to Q\in L^u(\pi;Q)$.
Recall $R$ is commutative and all $\cdot$'s are $R$-bilinear.  So for all $\bar{z}\in L^{s+t+u}(\pi)$, following
\eqref{def:module} we find the following holds.  
\begin{align*}
	(\bar{z})((\bar{x}\otimes r)\circ(\bar{y}\otimes s) \cdot f)
		& = -((\bar{x}\circ\bar{y})\circ\bar{z})f rs\\
		& = -(\bar{x}\circ(\bar{y}\circ \bar{z}))(fs)r+(\bar{y}\circ(\bar{x}\circ\bar{z}))(f r)s\\
		& = (\bar{y}\circ \bar{z})((\bar{x}\otimes r)\cdot (fs))-(\bar{x}\circ\bar{z})(\bar{y}\otimes s\cdot (fr))\\
		& = -(\bar{z})(\bar{y}\otimes s)\cdot ((\bar{x}\otimes r)\cdot f)+(\bar{z})(\bar{x}\otimes r)\cdot (\bar{y}\otimes s\cdot f).
\end{align*}
Hence we have
\begin{align*}
	((\bar{x}\otimes r)\circ (\bar{y}\otimes s)) \cdot f
		& = (\bar{x}\otimes r)\cdot ((\bar{y}\otimes s)\cdot f)-(\bar{y}\otimes s)\cdot ((\bar{x}\otimes r)\cdot f).
\end{align*}
So in $L^*(\pi;Q)$ is a graded Lie module for $L_*(\phi;R)$.
\hfill $\Box$

\section{Proof of Theorem~\ref{thm:auto}}
Now we prove our second claim, Theorem~\ref{thm:auto}, which claims that for an $A$-group $G$ with
an $A$-invariant filter $\phi:M\to 2^G$, there is a naturally induced filter $\Delta\phi:M\to 2^A$ as follows.
\begin{align*}
	\Delta_s\phi & = \{ a\in A : \forall t\in M, [\phi_t,a]\leq \phi_{t+s}\}.
\end{align*}
It further follows that $L_*(\Delta\phi)$ maps naturally into $\Der(L_*(\phi))$.

 Fix $s\in M$.  
Let $a,b\in \Delta_s \phi$.  So for all $t\in M$, and all $x\in \phi_t$, 
$[x,a^{-1} b]=[x^{-1},a]^{-b}[x,b]\in [\phi_t,a]^b [\phi_t,b]\leq \phi_{t+s}$ (using also the fact that $\phi$ is $A$-invariant).  
So $\Delta_s\phi$ is a subgroup of $A$.  Indeed, if $a\in A$ then as $\phi_t$ is $A$-invariant, $[\phi_t,a]\leq \phi_t=\phi_{t+0}$.  
So $a\in \Delta_0\phi$ proving that $A=\Delta\phi$.

Now suppose $s,t\in M$.  For each $u\in M$, $a\in \Delta_s\phi$, $b\in \Delta_t\phi$, $x\in\phi_u$ we find
\begin{align*}
	[x,[a,b]]  & = [b,[x^{-1},a^{-1}]]^{-xa} [a^{-1},[b^{-1},x]]^{-ba} \\
		& = [[x^{-1},a^{-1}],b]^{xa} [[b^{-1},x],a^{-1}]^{ba} \\
		& \in [\phi_{s+u},b] [\phi_{t+u},a] \leq \phi_{s+t+u}.
\end{align*}
So $[\phi_u,[\Delta_s \phi,\Delta_t \phi]]\leq \phi_{s+t+u}$.  Hence,
\begin{align*}
	[\Delta_s\phi,\Delta_t \phi] & \leq \Delta_{s+t}\phi.
\end{align*}

Finally, suppose that $s,t\in M$ and that $s\prec t$.
Take $a\in \Delta_t \phi$.  For each $u\in M$, $s+u\prec t+u$ and so $\phi_{s+u}\geq \phi_{t+u}$.  Hence,
$[\phi_u,a]\leq \phi_{t+u}\leq \phi_{s+u}$ which proves $a\in \Delta_s\phi$.  So $\Delta_s\phi\geq \Delta_t \phi$.
This proves $\Delta\phi$ is a filter.

Next we show that for each $s\in M-0$, $\Delta_s \phi$ acts on each $L_u(\phi)=\phi_u/\partial_u \phi$ as the identity.
For taking $a\in \Delta_s \phi$, $[\phi_u,a]\leq \phi_{u+s}\leq \partial_u \phi$.  Now define an action of
$\Delta_s \phi$ on $L_*(\phi)$ so that for each $u\in M-0$ and $a\in \Delta_s \phi$, 
$\bar{x}D_a= \overline{[x,a]}$ mapping $L_u(\phi)\to L_{s+u}(\phi)$.  First we observe that $[\phi_u,\Delta_s\phi]\in \phi_{u+s}$ and $[\partial_u \phi,\Delta_s\phi]\leq \partial_{s+u}\phi$ so that
$[yx,a]=[y,a]^x [x,a]\equiv [x,a] \pmod{\partial_{s+u} \phi}$.  So for each $u$,
$D_a:L_u(\phi)\mapsto L_{u+s}(\phi)$. Next take $x,y\in \phi_u$.  As $[x,a]\in \phi_{s+u}$ it follows that $[[x,a],y]\in \phi_{s+2u}\leq \partial_{s+u}\phi$.
Hence,
$(\bar{x}+\bar{y})D_a\equiv [xy,a]\equiv [x,a]^y[y,a]
\equiv [x,a] [[x,a],y] [y,a] \equiv [x,a][y,a]$ modulo $\partial_{s+u}\phi$.
Therefore $D_a$ is an additive endomorphism of $L_*(\phi)$.

Last we show that $D_a$ is derivation of $L_*(\phi)$.  For $u\in M-0$ and
$x,y\in \phi_u$,
\begin{align*}
	(\bar{x}\circ \bar{y})D_a & \equiv
		[[x,y],a] \equiv  [[x^{-1},a]^{-1},y^{-1}]^{-xy}[x,[y^{-1},a^{-1}]]^{ay} \pmod{\partial_{s+u}\phi}\\
	& \equiv [[x^{-1},a]^{-1},y^{-1}]^{-1}[x,[y^{-1},a^{-1}]]\\
	& \equiv \bar{x}D_a\circ\bar{y}+\bar{x}\circ\bar{y}D_a.
\end{align*}
So $D_a$ is a derivation on $L_*(\phi)$ shifting the grading by $s$.
\hfill $\Box$

\section{Examples}\label{sec:closing}

We close with some examples. In particular we demonstrate how filters and layerings
arise naturally within arbitrary groups and are not confined solely to nilpotent contexts.
\medskip

Consider first the case when $G=H\times K$.
Evidently $\gamma_i(H\times K)=\gamma_i(H)\times \gamma_i(K)$ and
$\zeta^j(H\times K)=\zeta^j(H)\times \zeta^j(K)$.  Using an acyclic monoid 
we get at the structure of $H$ and $K$ independently.  E.g. use $M=\mathbb{N}^2$ and
\begin{align*}
	\phi_{(a,b)} & = \gamma_a(H)\times \gamma_b(K) & \pi^{(a,b)} & = \zeta^a(H)\times \zeta^b(K).
\end{align*}
It follows that $\phi$ is a filter sifting the layering $\pi$.  More generally
given groups $\mathcal{H}$ and each $H\in\mathcal{H}$ has a filter $\phi^H:M_H\to 2^H$
sifting a layering $\pi_H:M\to 2^H$ then we obtain a filter 
$\phi:\prod_{H\in \mathcal{H}} M_H\to 2^{\prod\mathcal{H}}$ sifting
a layering $\pi:\prod_{H\in \mathcal{H}} M_H\to 2^{\prod\mathcal{H}}$ where
\begin{align*}
	\phi_m & = \prod_{H\in\mathcal{H}} \phi^H_{m(H)} &
	\pi_m & = \prod_{H\in\mathcal{H}} \pi_H^{m(H)}.
\end{align*}
We call this a {\em product filter} and respectively a {\em product layering}.
\smallskip

\subsection{Semi-classical filters}\label{sec:semi-classical}
Next suppose that $G$ is an arbitrary finite or pro-finite group.
In particular such groups have Sylow $p$-subgroups and also $p$-cores $O_p(G)$,
i.e. the intersection of all Sylow $p$-subgroups of $G$.
Given a set $\mathcal{P}$ of primes, the $\mathcal{P}$-core $O_{\mathcal{P}}(G)$
is a product of the $p$-cores for $p\in \mathcal{P}$.  In the case where $\mathcal{P}$
consists of all the primes we have the usual Fitting subgroups.  In any case 
we can create a product filter or layering.  
A natural example of this is to use the exponent-$p$ lower
central series of each $O_p(G)$.  As a result in that example the homogeneous
components of the associated Lie products are vector spaces over $\mathbb{Z}_p$
for the various $p\in \mathcal{P}$.  If it is more important that the filter
remain a series we can apply an order to the primes and then use the lexicographic
order on $\mathbb{N}^{\mathcal{P}}$.  A cyclic subfilter of a highly related form
was introduced by Eick-Horn called the $F$-series \cite{EH}*{Section 2}.

In order to compare these somewhat expected filters and layerings with more surprising
constructions we will describe the above constructions as {\em semi-classical}.  
While they need not have cyclic gradings they are all characterized by passing through the Fitting subgroup (or begin
at the nilpotent residual in the case of layerings) and are then built from 
classical filters of the primary components of a residually nilpotent group.

\subsection{Some non-semi-classical filters}\label{sec:refine}

Now we turn our attention to non-semi-classical filters and explain more carefully
what we are reporting in Figure~\ref{fig:data}.

In \cite{Wilson:alpha}*{Section 4} the three rings where 
introduced as tools to create filers (and now also layerings).
There are two further rings more compatible with decomposing the composition
products of layerings so we now include those as well.

Suppose we start with a known filter $\phi$ sifting a layering $\pi$.
We therefore have various bimaps $\circ:U\times V\bmto W$ from the products
between the homogeneous components of the Lie ring $L_*(\phi;R)$ and its module
$L^*(\pi;Q)$.  For any bimap we can define the following nonassociative rings.

First we have a Lie ring of derivations
\begin{align*}
	\Der(\circ) & = \{(f,g;h)\in{\frak gl}(U)\times{\frak gl}(V)\times {\frak gl}(W):
		 (uf)\circ v+u\circ (vg) = (u\circ v)h\}.
\end{align*}
Next we have three unital associative rings of left, mid, and right scalars.
(Here we use $\End(U)^{\bullet}$ to denote the opposite ring, equivalently
endomorphisms acting on the left.)
\begin{align*}
	\mathcal{L}(\circ) & = \{ (f,g)\in \End(U)^{\bullet}\times \End(W)^{\bullet} :
		 \forall u\forall v, (fu)\circ v=g(u\circ v)\},\\
	\mathcal{M}(\circ) & = \{ (f,g)\in \End(U)\times \End(V)^{\bullet} :
		 \forall u\forall v, (uf)\circ v= u\circ (gv)\},\\
	\mathcal{R}(\circ) & = \{ (f,g)\in \End(V)\times \End(W) : 
		\forall u\forall v, u\circ(vf)=(u\circ v)g\}.
\end{align*}
Finally we have an associative commutative unital ring called the {\em centroid}:
\begin{align*}
	\Cent(\circ) & = Z(\{(f,g;h)\in\End(U)\times \End(V)\times \End(W):\\
	& \qquad	 (uf)\circ v=(u\circ v)h=u\circ (vg)\}).
\end{align*}
Here $Z$ means to take the center of the ring.  Under very mild assumptions
on $\circ$ the conditions on $(f,g;h)$ already force commutativity.
There is a valuable interplay between these rings and automorphisms explained
in \cite{Wilson:beta}.

In \cite{Wilson:alpha}, $\mathcal{M}(\circ)$ was introduced as the ring of adjoints
and denoted $\Adj(\circ)$.  The rings $\mathcal{L}(\circ)$ and $\mathcal{R}(\circ)$
did not appear.  In any case one recognizes that these rings are each determined
by systems of linear equations.  Thanks to algorithms of Ronyai, Giani-Miller-Traeger,
Ivanyos, deGraaf, and many others, we can further compute such 
structure as Jacobson radicals, Levis subalgebras, Wedderburn-Mal'cev decompositions
etc.; cf. \citelist{\cite{deGraaf}\cite{IR:tapas}}.  Thus the action of these rings make $U$,
$V$, and $W$ into modules in various ways and the structure of the rings implies
characteristic structure of the modules.  Furthermore, the rings are compatible
with the original product $\circ$, e.g. $\circ$ is right $\mathcal{R}(\circ)$-linear -- which makes it ideal in decomposing the products $L_s(\phi;R)\times L^t(\pi;Q)\bmto L^{s+t}(\pi;Q)$. So the characteristic submodules of the various $U$, $V$,
and $W$ can be lifted back to characteristic subgroups of our original group.  See
\cite{Wilson:alpha} for expanded details and \cite{Maglione:filter} for some 
carefully computed examples.

\subsection{Data from small groups}
It had been shown in 
\cite{Wilson:alpha}*{Theorem 4.9} that using the ring $\mathcal{M}(\circ)$, there
is a positive proportion of all finite groups which admit a proper non-semi-classical
filter refining a semi-classical filter.  Although log-scale is the only option by
today's estimation techniques (see \cite{BNV:enum}*{Chapter 1}), it still could
be the case that in actual proportions very few groups benefit from this process.

In \cite{Maglione:filter}, Maglione created a series of algorithms for the computer
algebra system Magma which permit the physical construction of the $\mathcal{M}(\circ)$ filter refinements  described in Section~\ref{sec:refine}.  This permitted us to
conduct an experiment to measure the proportions of groups where refinement is found.
Adapting that code to also check for refinements by the other four rings we were 
astounded by the results.  Figure~\ref{fig:data} shows the proportion of groups
of order at most 1000 which have not only a non-classical Lie product, but in fact
a non-semi-classical Lie product.  While certainly in some cases the newly recovered
characteristic subgroups will have been knowable by other means we hasten to
mention that most of these groups (97\%) are nilpotent of class $2$ and have no
other standard characteristic subgroups other then the center and commutator.
So we are locating structure that is genuinely new.
Table~\ref{tab:tools} shows a breakdown by ring types for groups of
prime power order.  
\begin{table}
\begin{tabular}{|c||cccc|}
\hline
$p\backslash n$ & 4 & 5 & 6 & 7 \\
\hline\hline
 $2$ & 57\%	& 75\%	& 80\%	& 97\% \\ 	
 $3$ & 60\%	& 70\%	& 86\%	& 93\% \\ 	 
 $5$ &  60\% & 68\%	& 88\%	& 86\% \\ 	
 $7$ &  60\% & 67\%	& 88\%	& 81\% \\ 	
 $541$ &  60\%	& 51\% & 98\% 	& -- \\ 	
\hline
\end{tabular}
\caption{Lower bound on the proportion of $p^n$ order groups having a non-semi-classical refinement.}
\end{table}

\begin{table}
\begin{tabular}{|c||cccccc|}
\hline
$N$ & $2^4$ & $2^5$ & $2^6$ & $2^7$ & $2^8$ & $2^9$\\
Total $\%$ & 57\% & 75\% & 80\% & 97\% & 97\% & 79\%\\
\hline\hline
$\%$ by $\mathcal{M}(\circ)$ & 43\% & 51\% & 63\% & 81\% & 68\% & 37\%\\
$\%$ by $\Der(\circ)$ & 57\% & 75\% & 76\% & 92\% & 95\% & 79\% \\
$\%$ by $\Cent(\circ)$& 57\% & 57\% & 60\% & 65\% & 31\% & 5.4\%\\
\hline
\end{tabular}
\caption{Lower bounds on the proportions of groups of order $2^n$ having
non-semi-classical refinements by various methods.}
\label{tab:tools}
\end{table}

\section{Closing remarks}\label{sec:closing}

We have emphasized groups but the ideas generalize to nonassociatve rings
and also to loops by first using the work of \cite{Mostovoy} to obtain an initial 
associated graded ring.  From there the source for new gradings will likely 
benefit from the rings described in Section~\ref{sec:refine}. There is enormous symmetry within these rings and often a discovery
by one ring can be found by another, but we have examples within the small
group data above which show each is needed.  

We close with some speculations as to why the proportions seen in Figure~\ref{fig:data}
are so high.  First, if a group $G$ has a non-semi-classical
refinement then so does any central extension of $G$.  This explains some of the
proportions for $p$-groups of class at least 3.  
For $p$-groups of class 2 we 
are less certain.  We can take very large random polycyclic presentations of
$p$-groups of class $2$ and our methods regularly find nothing of interest.  
On the other hand in another model of randomness, such as taking quotients or 
subgroups of upper triangular matrices we find almost all such examples
admit non-semi-classical Lie product.  So the likelihood
of encountering examples with interesting structure depends greatly on the origin
of the groups of interest.  
Finally for non-nilpotent groups what seems a likely explanation
(and occurs in random examples) is that they have relatively large Fitting subgroups
and thus the proportions for nilpotent groups have a strong influence.  But this
is not the whole picture as for example 99\% of the $\approx$ 1 million groups of order $768=2^8\cdot 3$ have a non-semi-classical Lie product which exceeds
the proportion of nilpotent groups of lower order.  Whatever the reasons, 
the methods above are pliable and can be used to describe other future 
complex interactions between subgroups.  So it is not important that they remain focused on the rings
given in Section~\ref{sec:refine}.

\end{document}